\documentclass[a4paper, 10pt]{article}

\usepackage{graphics,graphicx}
\usepackage{amssymb,amsmath}
\usepackage{enumerate}

\usepackage{color}

\newtheorem{theorem}{Theorem}
\newtheorem{claim}{Claim}[theorem]

\newtheorem{lemma}{Lemma}

 \def \sm {\setminus}
 \def \es {\emptyset}

\newenvironment{proof}[1][]%
{\noindent {\setcounter{equation}{0}\it Proof.
}{#1}{}}{\hfill$\Box$\vspace{2ex}}

\newenvironment{proof2}[1][]%
{\noindent {\setcounter{equation}{0}\it Proof.
}{#1}{}}{\hfill$\Diamond$\vspace{2ex}}

\begin{document}

\title{Coloring graphs with no induced five-vertex path or gem\thanks{Dedicated to the memory of Professor~Fr\'ed\'eric Maffray.}}

\author{M.~Chudnovsky\thanks{Princeton University, Princeton, NJ 08544, USA.
Supported by NSF grant DMS-1763817.
This material is based upon work supported in part by the U.~S.~Army  Research
Laboratory and the U.~S.~Army Research Office under    grant number
W911NF-16-1-0404.} \and T.~Karthick\thanks{Computer Science Unit, Indian Statistical
Institute, Chennai Centre, Chennai 600029, India.}
\and{P.~Maceli\thanks{Adelphi University, Garden City, NY 11530, USA.}}
\and%
Fr\'ed\'eric Maffray\thanks{Deceased on August 22, 2018.}}

\date{\today}

\maketitle

\begin{abstract}
For a graph $G$, let $\chi(G)$ and $\omega(G)$ respectively denote the chromatic number  and clique number of $G$.
   We give an explicit structural description of ($P_5$,\,gem)-free graphs, and  show that every such graph $G$ satisfies   $\chi(G)\le \lceil\frac{5\omega(G)}{4}\rceil$. Moreover, this bound is  best possible.

\end{abstract}

\section{Introduction}

All our graphs are finite and have no loops or multiple edges. For any integer $k$, a \emph{$k$-coloring} of a graph $G$ is a mapping
$\phi:V(G)\rightarrow\{1,\ldots,k\}$ such that any two adjacent vertices
$u,v$ in $G$ satisfy $\phi(u)\neq \phi(v)$.  A graph is \emph{$k$-colorable}
if it admits a $k$-coloring.  The \emph{chromatic number} $\chi(G)$ of
a graph $G$ is the smallest integer $k$ such that $G$ is
$k$-colorable.   A \emph{clique} in a graph
$G$ is a set of pairwise adjacent vertices, and the \emph{clique number} of $G$, denoted by $\omega(G)$, is the size of a maximum clique in $G$. Clearly $\chi(H)\ge \omega(H)$ for every induced subgraph $H$ of $G$.
A graph $G$ is \emph{perfect} if every induced subgraph $H$ of $G$
satisfies $\chi(H) = \omega(H)$. Following Gy\'arf\'as
\cite{Gyarfas}, we say that a class of graphs is \emph{$\chi$-bounded}
if there is a function $f$ (called a \emph{$\chi$-bounding function}) such that every member $G$ of the class
satisfies $\chi(G)\le f(\omega(G))$. Thus the class of perfect graphs is $\chi$-bounded with $f(x)=x$.

For any integer $\ell$ we let $P_\ell$ denote the path
on $\ell$ vertices and $C_\ell$ denote the cycle on $\ell$ vertices. The \emph{gem} is the graph that consists of a $P_4$ plus a vertex
adjacent to all vertices of the $P_4$. A \emph{hole} (\emph{antihole}) in a graph is an induced subgraph that is isomorphic to $C_\ell$ ($\overline{C_\ell}$) with $\ell \ge 4$, and $\ell$ is the
length of the hole (antihole). A hole or an antihole is \emph{odd} if $\ell$ is odd.
Given a family of graphs ${\cal F}$, a graph $G$ is \emph{${\cal
F}$-free} if no induced subgraph of $G$ is isomorphic to a member of
${\cal F}$; when ${\cal F}$ has only one element $F$ we say that $G$
is $F$-free; when $\cal{F}$ has two elements $F_1$ and $F_2$, we simply write $G$ is ($F_1,F_2$)-free instead of $\{F_1,F_2\}$-free.
    Here we
are interested on $\chi$-boundedness for the class of ($P_5$,\,gem)-free graphs.

Gy\'arf\'as \cite{Gyarfas}
showed  that the class of $P_t$-free graphs is $\chi$-bounded. Gravier et al.~\cite{GHM} improved Gy\'arf\'as's
bound slightly by proving that every $P_t$-free graph $G$ satisfies
$\chi(G) \le (t-2)^{\omega(G)-1}$. It is well known that every $P_4$-free graph is perfect.
The preceding result implies that every $P_5$-free graph $G$ satisfies $\chi(G)\le 3^{\omega(G)-1}$.
 The problem of
determining whether the class of $P_5$-free graphs admits a polynomial
$\chi$-bounding function remains open, and the known $\chi$-bounding
function $f$ for such class of graphs satisfies $c(\omega^2/\log w)\le
f(\omega)\le 2^{\omega}$; see \cite{KPT-P5}.  So the recent focus is on
obtaining  $\chi$-bounding functions for some classes of
$P_5$-free graphs, in particular, for ($P_5,H$)-free graphs, for various graphs $H$.
The first author and Sivaram \cite{Chud-Siva} showed that every ($P_5,C_5$)-free graph $G$ satisfies $\chi(G) \leq 2^{\omega(G)-1}$, and that every ($P_5$, bull)-free graph $G$ satisfies $\chi(G) \le \binom{\omega(G)+1}{2}$. Fouquet
et al.~\cite{Fouquet} proved that there are infinitely many
($P_5,\,\overline{P_5}$)-free graphs $G$ with $\chi(G)\ge
\omega(G)^{\alpha}$, where $\alpha= \log_2 5-1$, and that every
($P_5,\,\overline{P_5}$)-free graph $G$ satisfies $\chi(G)\le
{{\omega(G)+1}\choose{2}}$.  The second author with Choudum and Shalu \cite{CKS}
studied the class of ($P_5$,\,gem)-free graphs and showed that every
such graph $G$ satisfies $\chi(G)\le 4\omega(G)$.
Later Cameron, Huang and Merkel \cite{CHM} impvroved this result replacing
$4 \omega$ with $\lfloor \frac{3 \omega(G)} {2} \rfloor$.
In this paper we  establish the best possible bound, as follows.

\begin{theorem}\label{thm:54bound}
Let $G$ be a ($P_5$,\,gem)-free graph. Then $\chi(G) \le \lceil\frac{5\omega(G)}{4}\rceil$.   Moreover, this bound is tight.
\end{theorem}

The degree of a vertex in a graph $G$ is the number of vertices adjacent to
it.  The maximum degree over all vertices in $G$ is denoted by
$\Delta(G)$. Clearly  every graph $G$ satisfies $\omega(G)\le \chi(G)\le \Delta(G)+1$.
Reed~\cite{Reed} conjectured that every graph $G$ satisfies $\chi(G)
\leq \lceil\frac{\Delta(G) + \omega(G) +1}{2}\rceil$. Reed's conjecture is still
open in general.
It is shown in \cite{KM2018} that if a graph $G$ satisfies $\chi(G) \le \lceil\frac{5\omega(G)}{4}\rceil$, then $\chi(G) \leq
\lceil\frac{\Delta(G) + \omega(G) +1}{2}\rceil$. So by Theorem~\ref{thm:54bound}, we immediately have the following theorem.

\begin{theorem}\label{thm:reeds}
Let $G$ be a  ($P_5$,\,gem)-free graph. Then $\chi(G) \leq
\lceil\frac{\Delta(G) + \omega(G) +1}{2}\rceil$. Moreover, this bound is tight.
\end{theorem}

The bounds in Theorem~\ref{thm:54bound} and
in Theorem~\ref{thm:reeds} are tight on the following example.  Let
$G$ be a graph whose vertex-set is partitioned into five cliques $Q_1,
\ldots, Q_5$ such that for each $i\bmod 5$, every vertex in $Q_i$ is
adjacent to every vertex in $Q_{i+1}\cup Q_{i-1}$ and to no vertex in
$Q_{i+2}\cup Q_{i-2}$, and $|Q_i|=q$ for all $i$ ($q>0$).  It is easy
to check that $G$ is ($P_5$,\,gem)-free.  Moreover, we have
$\omega(G)=2q$, $\Delta(G)=3q-1$, and $\chi(G)\ge
\lceil\frac{5q}{2}\rceil$ since $G$ has no stable set of size~$3$.

\bigskip

Our proof of Theorem~\ref{thm:54bound} uses the structure theorem for ($P_5$,\,gem)-free graphs (Theorem~\ref{thm:struc}).
   Before
stating it we recall some definitions.

\medskip
Let $G$ be a graph with vertex-set $V(G)$ and edge-set $E(G)$. For any two subsets
$X$ and $Y$ of $V(G)$, we denote by $[X,Y]$, the set of edges that has
one end in $X$ and other end in $Y$.  We say that $X$ is
\emph{complete} to $Y$ or $[X,Y]$ is complete if every vertex in $X$
is adjacent to every vertex in $Y$; and $X$ is \emph{anticomplete} to
$Y$ if $[X,Y]=\emptyset$.  If $X$ is singleton, say $\{v\}$, we simply
write $v$ is complete (anticomplete) to $Y$ instead of writing $\{v\}$
is complete (anticomplete) to $Y$.   For any $x \in V(G)$, let $N(x)$ denote the set of all
neighbors of $x$ in $G$; and let $\mbox{deg}_G(x):=|N(x)|$.  The neighborhood $N(X)$ of a
subset $X \subseteq V(G)$ is the set $\{u \in V(G)\setminus X  \mid  u$ $
\mbox{~is adjacent to a vertex of }X\}$. If $X\subseteq V(G)$, then $G[X]$
denote the subgraph induced by $X$ in $G$.  A set $X \subseteq V(G)$ is a \emph{homogeneous set} if every vertex with a neighbor in $X$ is
complete to $X$. Note that in any  gem-free graph $G$, for every $v\in V(G)$, $N(v)$ induces a $P_4$-free graph, and hence the subgraph induced by a  homogeneous set in $G$ is $P_4$-free.

An \emph{expansion} of a graph $H$ is any graph $G$ such that $V(G)$ can
be partitioned into $|V(H)|$ non-empty  sets
$Q_v$, $v\in V(H)$, such that $[Q_u,Q_v]$ is complete if $uv\in E(H)$,
and $[Q_u,Q_v]=\es$ if $uv\notin E(H)$. An expansion of a graph is a \emph{clique expansion} if each $Q_v$ is a clique,
and is a \emph{$P_4$-free expansion} if each $Q_v$ induces a $P_4$-free graph.

\begin{figure}[h]
\centering
        \includegraphics[width=11cm]{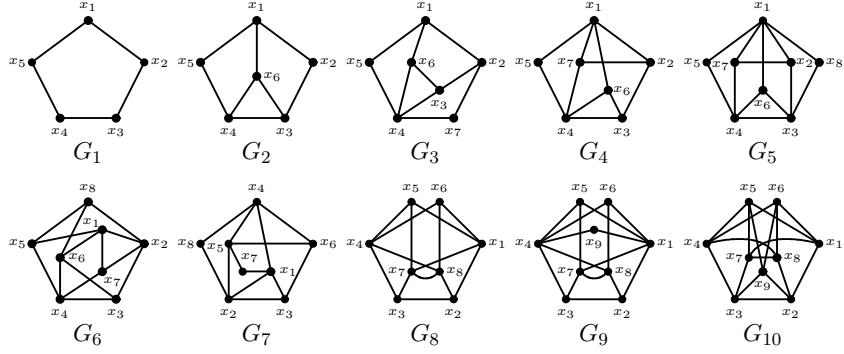}
\caption{Basic graphs}\label{fig:g1g10}
\end{figure}

Let $G_1,G_2,\ldots,G_{10}$ be the ten graphs  shown in Figure~\ref{fig:g1g10}.
 Clearly each of $G_1,\ldots,G_{10}$ is ($P_5$,\,gem)-free.  Moreover, it is
easy to check that any $P_4$-free expansion of a ($P_5$,\,gem)-free
graph is ($P_5$,\,gem)-free.

\medskip

 Let ${\cal H}$ be the class of connected ($P_5$,\,gem)-free graphs $G$ such
that $V(G)$ can be partitioned into seven non-empty sets $A_1, \ldots,
A_7$ such that:
\begin{itemize}
\item Each $A_i$ induces a $P_4$-free graph.
\item $[A_1,A_2\cup A_5\cup A_6]$ is complete and $[A_1,A_3\cup A_4\cup A_7]=\emptyset$.
\item $[A_3,A_2\cup A_4 \cup A_6]$ is complete and $[A_3,A_5\cup A_7]=\emptyset$.
\item
$[A_4, A_5\cup A_6]$ is complete and $[A_4, A_2\cup A_7]=\emptyset$.
\item
$[A_2, A_5\cup A_6\cup A_7]=\emptyset$ and $[A_5,A_6 \cup A_7]=\emptyset$.
\item The vertex-set of each component of $G[A_7]$ is a homogeneous set.
\item[]
(Adjacency between $A_6$ and $A_7$ is not specified, but it is
restricted by the fact that $G$ is ($P_5$,\,gem)-free.)
\end{itemize}

\smallskip
Now we can state our structural result.

\begin{theorem}\label{thm:struc}
Let $G$ be a connected ($P_5$,\,gem)-free graph that contains an induced $C_5$.
Then either $G\in \cal{H}$ or $G$ is a $P_4$-free expansion of either
$G_1$, $G_2$, \ldots, $G_9$ or $G_{10}$.
\end{theorem}
We note that  another structure theorem for ($P_5$,\,gem)-free graphs using a recursive construction is given by Brandst\"adt and Kratsch \cite{BK2005}. However, it seems difficult to use that theorem to get the bounds derived in this paper.

\section{Proof of Theorem~\ref{thm:struc}}
\label{sec:p5gem1}

Let $G$ be a connected $(P_5, gem)$-free graph.
Since $G$ contains an induced  $C_5$, there are five non-empty and pairwise
disjoint sets $A_1, ..., A_5$ such that for each $i$ modulo~$5$ the
set $A_i$ is complete to $A_{i-1}\cup A_{i+1}$ and anticomplete to
$A_{i-2}\cup A_{i+2}$.  Let $A=A_1\cup\cdots\cup A_5$.  We choose
these sets such that $A$ is maximal.  From now on every subscript is
understood modulo~$5$.  Let $R=\{x\in V(G)\setminus A \mid x$ has no
neighbor in $A\}$, and for each $i$ let:
\begin{eqnarray*}
Y_i &=& \{x\in V(G)\setminus A \mid x \mbox{ is complete to } A_i,
\mbox{ anticomplete to } A_{i-1}\cup A_{i+1}, \\
& & \mbox{ and $x$ has a neighbor in each of $A_{i-2}$ and
$A_{i+2}$, and $x$ is complete}\\
& & \mbox{ to one of $A_{i-2}$ and $A_{i+2}$}\}.
\end{eqnarray*}

\begin{claim}\label{ayr}
$V(G) = A_1\cup\cdots\cup A_5 \cup Y_1\cup\cdots\cup Y_5 \cup R$.
\end{claim}
\begin{proof2}
Consider any $x\in V(G)\setminus (A\cup R)$.  For each $i$ let
$a_i$ be a neighbor of $x$ in $A_i$ (if any such vertex exists) and
$b_i$ be a non-neighbor of $x$ in $A_i$ (if any exists).  Let
$L=\{i\mid a_i$ exists$\}$.  Then $L\neq\emptyset$ since $x\notin R$.
Up to symmetry there are four cases: \\
(a) $L=\{i\}$ or $\{i,i+1\}$ for some $i$.  Then
$x$-$a_i$-$b_{i-1}$-$b_{i-2}$-$b_{i-3}$ is a $P_5$, a
contradiction.  \\
(b) $L=\{i-1,i+1\}$ or $\{i-1,i,i+1\}$ for some $i$.  Then $x$ is
complete to $A_{i-1}\cup A_{i+1}$, for otherwise we find a $P_5$ as in
case~(a).  But then $x$ can be added to $A_i$, contradicting the
maximality of $A$.  \\
(c) $L=\{i, i-2, i+2\}$ for some $i$.  Then $x$ is complete to $A_i$,
for otherwise we find a $P_5$ as in case~(a), and similarly $x$ must
be complete to one of $A_{i-2}$ and $A_{i+2}$.  So $x$ is in $Y_i$.
\\
(d) $|L|\ge 4$.  Then $\{x, a_i, a_{i+1}, a_{i+2}, a_{i+3}\}$ induces
a gem for some~$i$, a contradiction.
\end{proof2}

\begin{claim}\label{ay}
For each $i$, $G[A_i]$ and
$G[Y_i]$ are $P_4$-free.
\end{claim}
\begin{proof2} Since $G$ is gem-free, the claim follows by the definitions of $A_i$ and $Y_i$.  \end{proof2}

\begin{claim}\label{yy1}
For each $i$ we have $[Y_{i-1}, Y_{i+1}]=\emptyset$.
\end{claim}
\begin{proof2}
Pick any $y\in Y_{i-1}$ and $z\in Y_{i+1}$.  We know that $y$ has neighbors $a_{i+1}\in A_{i+1}$ and
$a_{i+2}\in A_{i+2}$, and $z$ has a neighbor $a_{i-1}\in A_{i-1}$.
Then
$yz\notin E(G)$, for otherwise $\{a_{i-1}, z, a_{i+1}, a_{i+2}, y\}$ induces a gem, a contradiction.
\end{proof2}

\medskip

We say that a vertex in $Y_i$ is \emph{pure} if it is complete to
$A_{i-2}\cup A_{i+2}$, and the set $Y_i$ is \emph{pure} if every vertex in $Y_i$ is pure.

\begin{claim}\label{ExistPureAllpure}
Suppose that there exists a pure vertex in $Y_i$ for some $i$. Then $Y_i$ is pure.
\end{claim}
\begin{proof2} We may assume that $i=1$ and let $p\in Y_1$ be pure. Suppose to the contrary that there exists a vertex $y\in Y_1$ and is not pure, say $y$ has a non-neighbor $b_3\in A_3$. Moreover, by the definition of $Y_1$, $y$ has a neighbor $a_3\in A_3$.
Then $b_3a_3 \notin E(G)$,  for otherwise $b_3$-$a_3$-$y$-$a_1$-$a_5$ is a $P_5$ for any $a_1\in A_1$ and $a_5\in A_5$.
Also, for any $a_1\in A_1$ and $a_4\in A_4$, since $\{a_1,y,a_4,b_3,p\}$ does not induce a gem, we have $py\notin E(G)$. But, then for any $a_4\in A_4$, $\{b_3,p,a_3,y, a_4\}$ induces a gem,  a contradiction.
\end{proof2}

\begin{claim}\label{yiai+2ai-2}
For each $i$, we have: either $[Y_i,A_{i+2}]$ is complete or $[Y_i,A_{i-2}]$ is complete.
\end{claim}
\begin{proof2} We may assume that $i=1$. Suppose to the contrary that there exist vertices $y_1$ and $y_2$ in $Y_1$ such that $y_1$ has a non-neighbor $b_4 \in A_4$ and $y_2$ has a non-neighbor $b_3\in A_3$. By the definition of $Y_1$, $y_1$ is complete to $A_3$, and has a neighbor $a_4 \in A_4$. Likewise, $y_2$ is complete to $A_4$, and  has a neighbor $a_3\in A_3$. Then $a_3b_3\notin E(G)$, for otherwise $b_3$-$a_3$-$y_2$-$a_1$-$a_5$ is a $P_5$ for any $a_1\in A_1$ and $a_5\in A_5$. Also, for any $a_1\in A_1$, since $\{a_1,y_1,a_3,b_4,y_2\}$ does not induce a gem, we have $y_1y_2\notin E(G)$. But, then $\{b_3,y_1,a_3, y_2,a_4\}$ induces a gem, a contradiction.
\end{proof2}

\begin{claim}\label{Green-yi}   Suppose that $[Y_i,A_{i-2}]$ is complete for some $i$. Let  $A'_{i+2}=N(Y_i)\cap A_{i+2}$ and
$A''_{i+2}= A_{i+2}\setminus A'_{i+2}$. Then (i) $[A'_{i+2}, A''_{i+2}] =\emptyset$, and (ii) $[Y_i, A'_{i+2}]$ is complete.
\end{claim}
\begin{proof2}
$(i)$: Suppose to the contrary that there are adjacent vertices $p\in A'_{i+2}$ and $q\in A''_{i+2}$. Pick a neighbor of $p$ in $Y_i$, say $y$. Clearly $yq\notin E(G)$. Then for any $a_i\in A_i$ and $a_{i-1}\in A_{i-1}$, $q$-$p$-$y$-$a_i$-$a_{i-1}$ is a $P_5$, a contradiction. This proves item $(i)$.

\noindent{$(ii)$}: Suppose to the contrary that there are non-adjacent vertices $y\in Y_i$ and $p\in A'_{i+2}$. Pick a neighbor of $p$ in $Y_i$, say $y'$. By the definition of $Y_i$, $y$ has a neighbor in $A'_{i+2}$, say $q$. Pick any $a_{i-2}\in A_{i-2}$, $a_{i-1}\in A_{i-1}$ and $a_i\in A_i$. Now, $pq\notin E(G)$, for otherwise $p$-$q$-$y$-$a_i$-$a_{i-1}$ is a $P_5$. Also, $yy'\notin E(G)$, for otherwise $\{p,a_{i-2},y,a_i,y'\}$ induces a gem. Then since $\{y,q,y',p,a_{i-2}\}$ does not induce a gem, $qy'\notin E(G)$. But then $p$-$y'$-$a_i$-$y$-$q$ is a $P_5$, a contradiction. This proves item $(ii)$.
\end{proof2}

\begin{claim}\label{yy2}
Suppose that $Y_{i-2}$ and $Y_{i+2}$ are both non-empty for
some $i$.  Let $A_i^- = N(Y_{i-2})\cap A_i$ and $A_i^+ =
N(Y_{i+2})\cap A_i$.  Then:
\begin{enumerate}[(a)]\setlength\itemsep{0em}
\item $[Y_{i-2}, Y_{i+2}]$ is complete, $A_i^-\cap A_i^+=\emptyset$, and $[A_i^-,A_i^+]=\emptyset$,
\item
$[A_i\setminus (A_i^-\cup A_i^+), A_i^-\cup A_i^+]=\emptyset$,
\item
 $[Y_{i-2}, A_{i+1}\cup
A_i^-]$ and $[Y_{i+2}, A_{i-1}\cup A_i^+]$ are complete,

\item
$Y_{i-1}\cup Y_{i+1}=\emptyset$,
\item
Any vertex in $Y_i$ is pure,
\item
One of the sets $A_i\setminus (A_i^-\cup A_i^+)$ and $Y_i$ is empty.
\end{enumerate}
\end{claim}
\begin{proof2} Pick any $y\in Y_{i-2}$ and $z\in Y_{i+2}$.  So $y$ has
neighbors $a_{i-2}\in A_{i-2}$, $a_{i+1}\in A_{i+1}$ and $a_i\in A_i$,
and $z$ has neighbors $a_{i+2}\in A_{i+2}$, $a_{i-1}\in A_{i-1}$ and
$b_i\in A_i$.

\noindent{$(a)$}:   Now $yz\in E(G)$, for otherwise
$y$-$a_{i+1}$-$a_{i+2}$-$z$-$a_{i-1}$ is a  $P_5$.  Since this
holds for arbitrary $y,z$, we obtain that $[Y_{i-2}, Y_{i+2}]$ is
complete.  Then $za_i\notin E(G)$, for otherwise $\{a_i, a_{i+1}, y,
z, a_{i-1}\}$ induces a gem, and similarly $yb_i\notin E(G)$.  In
particular $a_i\neq b_i$; moreover $a_ib_i\notin E(G)$, for otherwise
$a_i$-$b_i$-$z$-$a_{i+2}$-$a_{i-2}$ is a $P_5$.  Since this
holds for any $y,z,a_i,b_i$, it proves item~(a).

\noindent{$(b)$}: Suppose that there are adjacent vertices $u\in A_i\setminus (A_i^-\cup A_i^+)$
and $v\in A_i^-\cup A_i^+$, say $v\in A_i^-$. Then $u$-$v$-$y$-$a_{i-2}$-$a_{i+2}$
is a  $P_5$, a contradiction. This proves item~(b).

 \noindent{$(c)$}:  Since $y$ and $z$ are not complete to $A_i$ (by $(a)$), by Claim~\ref{yiai+2ai-2}, $[Y_{i-2}, A_{i+1}]$ and
 $[Y_{i+2}, A_{i-1}]$ are complete. Also, by Claim~\ref{Green-yi}(ii), $[Y_{i-2}, A_i^-]$ and $[Y_{i+2}, A_i^+]$ are complete.
 This proves item~(c).

 \noindent{$(d)$}: If $Y_{i-1}\neq\emptyset$ then, by a similar argument as in the proof of $(c)$ (with subscripts
shifted by~$1$), $[Y_{i-2}, A_i]$ should be complete, which it is not.  So
$Y_{i-1}=\emptyset$, and similarly $Y_{i+1}=\emptyset$.  This proves item~(d).

 \noindent{$(e)$}: Consider any $x\in Y_i$ and suppose that it is not pure; up
to symmetry $x$ has a non-neighbor $b\in A_{i+2}$ and   is
complete to $A_{i-2}$.  By Claim~\ref{yy1} we know that $xz\notin E(G)$.
Then $a_i$-$x$-$a_{i-2}$-$b$-$z$ is a  $P_5$.  This proves
item~(e).

\noindent{$(f)$}: Suppose that there are vertices $b\in A_i\setminus (A_i^-\cup A_i^+)$ and $u\in Y_i$.
By the definition of $Y_i$, we know that $bu\in E(G)$, and by Claim~\ref{yy1}, $uy, uz\notin E(G)$. Then  by item~(c) and item~(e), for any $a_{i-2}\in A_{i-2}$,  $b$-$u$-$a_{i-2}$-$y$-$z$ is a  $P_5$, a contradiction. This proves item~(f).

\end{proof2}

\begin{claim}\label{rh}
The vertex-set of each component of $G[R]$ is a homogeneous set, and hence each component of $G[R]$ is $P_4$-free.
\end{claim}
\begin{proof2} Suppose that a vertex-set of a component $T$ of $G[R]$ is not homogeneous.
Then there are adjacent vertices $u,t\in V(T)$ and a vertex $y\in
V(G)\setminus V(T)$ with $yu\in E(G)$ and $yt\notin E(G)$.
By~Claim~\ref{ayr} we have $y\in Y_i$ for some $i$.  Then
$t$-$u$-$y$-$a_i$-$a_{i+1}$ is a  $P_5$, for any $a_i\in A_i$ and
$a_{i+1}\in A_{i+1}$, a contradiction.
\end{proof2}

\begin{claim}\label{ry}
Suppose that there is any edge $ry$ with $r\in R$ and $y\in Y_i$. Then
$y$ is pure and $Y_{i-1}\cup Y_{i+1}=\emptyset$.  Moreover if any of
the sets $Y_{i-2}, Y_{i+2}$ is non-empty then exactly one of them is
non-empty, and $R$ is complete to that non-empty set and to $Y_i$.
\end{claim}
\begin{proof2}
Consider any edge $ry$ with $r\in R$ and $y\in Y_i$.  So $y$
has a neighbor $a_j\in A_j$ for each $j\in\{i, i-2, i+2\}$.  If $y$ is
not pure, then up to symmetry $y$ has a non-neighbor $b\in A_{i-2}$,
and then $r$-$y$-$a_i$-$a_{i-1}$-$b$ is a  $P_5$ for any
$a_{i-1}\in A_{i-1}$, a contradiction.  So $y$ is pure, and
by~Claim~\ref{yy2} $Y_{i-1}\cup Y_{i+1}=\emptyset$.  Now suppose up to
symmetry that there is a vertex $z\in Y_{i+2}$.  By~Claim~\ref{yy1} we
have $yz\notin E(G)$.  Then $rz\in E(G)$, for otherwise
$r$-$y$-$a_{i+2}$-$z$-$a_{i-1}$ is a  $P_5$, for any
$a_{i-1}\in A_{i-1}\cap N(z)$.  Now by the same argument as above, $z$
is pure, and by~Claim~\ref{yy2} $Y_{i+1}\cup Y_{i+3}=\emptyset$.
Since this holds for any $z$, the vertex $r$ is complete to $Y_{i+2}$,
and then by symmetry $r$ is complete to $Y_i$; and by Claim~\ref{rh} and
the fact that $G$ is connected, $R$ is complete to $Y_i\cup Y_{i+2}$.
\end{proof2}

\medskip

It follows from the preceding claims that at most three of the sets $Y_1,...,Y_5$ are non-empty,
and   if $R\neq\emptyset$ then at most two of $Y_1, ..., Y_5$ are non-empty.  Hence we
have the following cases:
\begin{enumerate}[(A)]\setlength\itemsep{0em}

\item
$R=\emptyset$ and $Y_2\cup Y_3\cup Y_5=\emptyset$. Any of $Y_1, Y_4$ may
be non-empty.\\
We may assume that both $Y_1$ and $Y_4$ are not empty, $[Y_1, A_3]$ is complete and $[Y_4, A_2]$ is complete. (Otherwise, using Claims~~\ref{ay}, \ref{yiai+2ai-2} and \ref{Green-yi}, it follows that $G$ is a   $P_4$-free expansion of either $G_1,G_2,\ldots, G_6$ or $G_9$.)
Suppose there exists $y_1\in Y_1$ that has a non-neighbor $a_4 \in A_4$, and there exists $y_4\in Y_4$ that has a non-neighbor $a_1\in A_1$, then for any $a_3\in A_3$, $a_1$-$y_1$-$a_3$-$a_4$-$y_4$ is a $P_5$ in $G$, a contradiction. So either  $Y_1$ is pure or  $Y_4$ is pure.
Then by Claims~\ref{ay},~\ref{yiai+2ai-2} and \ref{Green-yi}, we see that   $G$ is a   $P_4$-free expansion of  $G_4$ or $G_6$.
\item
$R=\emptyset$ and $Y_2, Y_3$ are both non-empty.\\
 Then
Claims~\ref{ay} and \ref{yy2} implies that $G$ is a $P_4$-free expansion of either $G_8$, $G_9$ or $G_{10}$.
\item
$R\neq\emptyset$ and exactly one of $Y_1, ..., Y_5$ is non-empty, say $Y_1$ is non-empty.\\ In this case, we show that $G\in \cal{H}$ as follows: Since $R\neq \emptyset$, there exists a vertex $r\in R$ and a vertex
$y\in Y_1$ such that $ry \in E(G)$. Then by Claim~\ref{ry}, $y$ is a pure vertex of $Y_1$. So, by Claim~\ref{ExistPureAllpure}, $Y_1$ is pure, and hence by Claims~\ref{ay} and \ref{rh}, we see that $G\in \cal{H}$.
\item
$R\neq\emptyset$ and exactly two of $Y_1, ..., Y_5$ are non-empty.\\
In this case, by Claims~\ref{rh}
and~\ref{ry} and up to symmetry we may assume that $Y_1$ and $Y_4$
are non-empty, all vertices in $Y_1\cup Y_4$ are pure, and $[R,
Y_1\cup Y_4]$ is complete. Moreover, since $G$ is gem-free,  $G[R]$ is $P_4$-free. So by Claim~\ref{ay}, $G$ is a $P_4$-free expansion  of $G_7$.

\end{enumerate}
This completes the proof of  Theorem~\ref{thm:struc}. \hfill{$\Box$}

\section{Bounding the chromatic number}

 A
\emph{stable set} is a set of pairwise non-adjacent vertices.
  We say that two sets \emph{meet} if their
intersection is not empty.
In a graph $G$, we say that a stable set is \emph{good} if it meets
every clique of size $\omega(G)$. Moreover, we say that a clique $K$ in $G$ is a \emph{$t$-clique} of $G$ if $|K|=t$.

We  use the following theorem often.
\begin{theorem}[\cite{KM-arxiv}]\label{thm:tools}
Let $G$ be a graph such that every proper induced subgraph $G'$ of $G$
satisfies $\chi(G')\le \lceil \frac{5}{4}\omega(G')\rceil$.  Suppose
that one of the following occurs:

\begin{enumerate}[(i)]\itemsep=1pt
\item\label{degq}
$G$ has a vertex of degree at most
$\lceil\frac{5}{4}\omega(G)\rceil-1$.
\item\label{goods}
$G$ has a good stable set.
\item\label{stable}
$G$ has a stable set $S$ such that $G\sm S$ is perfect.
\item\label{fives}
For some integer $t\ge 5$ the graph $G$ has $t$ stable sets
$S_1,\ldots,S_t$ such that $\omega(G\sm (S_1\cup\cdots\cup S_t))\le
\omega(G)-(t-1)$.
\end{enumerate}
Then $\chi(G)\le \lceil \frac{5}{4}\omega(G)\rceil$.
\end{theorem}

Given a graph $G$ and a proper homogeneous set $X$ in $G$, let $G/X$
be the graph obtained by replacing $X$ with a clique $Q$ of size
$\omega(X)$ (i.e., $G/X$ is obtained from $G\setminus X$ and $Q$ by
adding all edges between $Q$ and the vertices of $V(G)\setminus X$
that are adjacent to $X$ in $G$).

\begin{lemma}[\cite{KM2018}]\label{lem:red}
In a graph $G$ let $X$ be a proper homogeneous set such that $G[X]$ is
$P_4$-free.  Then $\omega(G)=\omega(G/X)$ and  $\chi(G)=\chi(G/X)$.  Moreover, $G$ has a good stable set if
and only if $G/X$ has a good stable set.
\end{lemma}

 For $k \in \{1, 2, \ldots, 10\}$, let ${\cal G}_k$ be the class of
graphs that are $P_4$-free expansions of $G_k$, and let ${\cal G}_k^*$
be the class of graphs that are clique expansions of $G_k$. Let $\cal{H}^*$ be the class of graphs $G \in \cal{H}$ such that, with the notation as in Section~1, the five sets $A_1,A_2,\ldots, A_5$, and the vertex-set of each component of $G[A_7]$ are cliques.

The following lemma can be proved using Lemma~\ref{lem:red}, and the proof is very similar to that of Lemma~3.3  of \cite{KM2018}, so we omit the details.

\begin{lemma} \label{lem:red2}
For every graph $G$ in ${\cal G}_i$ ($i\in\{1, \ldots, 10\}$) (resp. $G$ in $\cal{H}$)
 there is a graph $G^*$ in ${\cal G}_i^*$
($i\in\{1, \ldots, 10\}$) (resp. $G^*$ in $\cal{H}^*$) such that
$\omega(G)=\omega(G^*)$ and $\chi(G)=\chi(G^*)$. Moreover, $G$ has a good stable set if and
only if $G^*$ has a good stable set.
\end{lemma}

By Lemma~\ref{lem:red2} and Theorem~\ref{thm:struc}, to prove Theorem~\ref{thm:54bound}, it suffices to consider the clique expansions of $G_1,G_2\ldots,G_{10}$ and the members of $\cal{H^*}$.

\subsection{Coloring clique expansions}
\begin{theorem}\label{thm:54-G1-G6}
Let $G$ be a clique expansion of either $G_1,\ldots, G_5$ or $G_6$, and assume that every induced subgraph
$G'$ of $G$ satisfies $\chi(G')\le \lceil\frac{5}{4}\omega(G')\rceil$.
Then $\chi(G)\le \lceil\frac{5}{4}\omega(G)\rceil$.
\end{theorem}
\begin{proof}
Throughout the proof of this theorem, we use the following notation:
 Let $q=\omega(G)$. Suppose that $G$ is a  clique expansion of $H\in \{G_1,\ldots,G_6\}$.
 So there is a partition of $V(G)$ into $|V(H)|$ non-empty cliques $Q_1, \ldots, Q_{|V(H)|}$, where $Q_i$ corresponds to
the vertex $x_i$ of $H$.  We write, e.g., $Q_{12}$ instead of
$Q_1\cup Q_2$, $Q_{123}$ instead of $Q_1\cup Q_2\cup Q_3$, etc.  For each
$i\in\{1,\ldots,|V(H)|\}$ we call $x_i$ one vertex of $Q_i$.  Moreover if
$|Q_i|\ge 2$ we call $x'_i$ one vertex of $Q_i\setminus\{x_i\}$.
Recall that if $G$ has a good stable set, then we can conclude the theorem using Theorem~\ref{thm:tools}(\ref{goods}).

\smallskip

 \noindent{(I)} Suppose that $G$ is a clique expansion of $G_1$.  (We refer to \cite{KM2018, KM-arxiv} for alternate proofs.) We may assume that $|Q_i|\geq 2$, for each $i\in \{1,\ldots,5\}$, otherwise if
 $|Q_1|=1$ (say), then  $G\setminus\{x_1\}$ is perfect, as it is a clique expansion of $P_4$, and we can conclude with Theorem~\ref{thm:tools}(\ref{stable}).
 Let $X$ be a subset of $V(G)$ obtained by taking
two vertices from $Q_i$ for each $i\in \{1,\ldots,5\}$. Then since $\chi(G[X])=5$ and $\omega(G\setminus X) = q-4$, by hypothesis, we have $\chi(G)\le \lceil\frac{5}{4}\omega(G\setminus X)\rceil+5 \le\lceil\frac{5}{4}q\rceil$.

\smallskip

\noindent (II) Suppose that $G$ is a  clique expansion of $G_2$. Then $\{x_2,x_5,x_6\}$ is a good stable set of $G$, and we can conclude with Theorem~\ref{thm:tools}(\ref{goods}).

\medskip
\noindent (III) Suppose that $G$ is a  clique expansion of $G_3$. Suppose that $|Q_5|\le |Q_6|$.  By hypothesis we can color $G\setminus
Q_5$ with $\lceil\frac{5}{4}q\rceil$ colors.  Since $Q_6$ is complete
to $Q_1\cup Q_4$, which is equal to $N(Q_5)$, we can extend this
coloring to $Q_5$, using for $Q_5$ the colors used for $Q_6$.
Therefore let us assume that $|Q_5| > |Q_6|$.  It follows that  $|Q_{15}|>|Q_{16}|$, so $Q_{16}$ is
not a $q$-clique.   Likewise we may assume that $|Q_7|>|Q_3|$, and consequently $Q_{23}$ is not a $q$-clique. Then:

$Q_{15}$ is a $q$-clique, for otherwise $\{x_2,x_4\}$ is a good stable set of $G$.

$Q_{45}$ is a $q$-clique, for otherwise $\{x_1,x_3,x_7\}$ is a good stable set of $G$.

$Q_{12}$ is a $q$-clique, for otherwise $\{x_3,x_5,x_7\}$ is a good stable set of $G$.

$Q_{47}$ is a $q$-clique, for otherwise $\{x_2,x_5,x_6\}$ is a good stable set of $G$.

$Q_{27}$ is a $q$-clique, for otherwise $\{x_1,x_4\}$ is a good stable set of $G$.

\noindent{}The above properties imply that there is an integer $a$ with $1\le a\le q-1$ such that
$|Q_2|=|Q_5|=|Q_7|=a$ and $|Q_1|=|Q_4|=q-a$. Since $|Q_7|>|Q_3|$, we have $a\geq 2$.
Since $q=|Q_{27}|=2a$, $a=\frac{q}{2}$. So $q$ is even, $q\ge 4$ and
$|Q_1|=|Q_2|=|Q_4|=|Q_5|=|Q_7| = \frac{q}{2}\ge 2$.

Now consider the five stable sets $\{x_1, x_3, x_7\}$, $\{x'_1, x_4\}$, $\{x_5, x_6, x'_7\}$, $\{x_2, x'_4\}$ and $\{x'_2,
x'_5\}$.  It is easy to see that their union $U$ meets every
$q$-clique  four times.  It follows  that $\omega(G\setminus U) = q-4$, and we can conclude
using Theorem~\ref{thm:tools}(\ref{fives}).

\medskip

 \noindent (IV) Suppose that $G$ is a  clique expansion of either $G_4$ or $G_5$. Suppose that $|Q_5|\le |Q_7|$.  By hypothesis we can color $G\setminus
Q_5$ with $\lceil\frac{5}{4}q\rceil$ colors.  Since $Q_7$ is complete
to $Q_1\cup Q_4$, which is equal to $N(Q_5)$, we can extend this
coloring to $Q_5$, using for $Q_5$ the colors used for $Q_7$.
Therefore let us assume that $|Q_5| > |Q_7|$.  It follows that  $|Q_{45}|>|Q_{47}|$, so $Q_{47}$ is
not a $q$-clique.   Likewise we may assume that $|Q_5|>|Q_6|$ (for otherwise any $\lceil\frac{5}{4}q\rceil$-coloring of
$G\setminus Q_5$ can be extended to $Q_5$), and consequently $Q_{16}$ is not a $q$-clique. Now if $G$ is a  clique expansion of $G_4$, then $\{x_2,x_5,x_6\}$ is a good stable set of $G$, and if $G$ is a  clique expansion of   $G_5$, then $\{x_2,x_5,x_6,x_8\}$ is a good stable set of $G$.
In either case, we can conclude the theorem with Theorem~\ref{thm:tools}(\ref{goods}).

\medskip
\noindent (V) Suppose that $G$ is a  clique expansion of $G_6$. Suppose that $|Q_8|\le |Q_1|$.  By hypothesis we can color $G\setminus
Q_8$ with $\lceil\frac{5}{4}q\rceil$ colors.  Since $Q_1$ is complete
to $Q_2\cup Q_5\cup Q_6$, which is equal to $N(Q_8)$, we can extend this
coloring to $Q_8$, using for $Q_8$ the colors used for $Q_1$.
Therefore let us assume that $|Q_8| > |Q_1|$.  It follows that  $|Q_{68}|>|Q_{16}|$ and $|Q_{58}|>|Q_{15}|$, and consequently
$Q_{16}$ and $Q_{15}$ are not $q$-cliques.   Likewise we may assume that $|Q_5|>|Q_6|$ (for otherwise any $\lceil\frac{5}{4}q\rceil$-coloring of
$G\setminus Q_5$ can be extended to $Q_5$), and consequently $Q_{68}$ is not a $q$-clique. Then:

$Q_{23}$ is a $q$-clique, for otherwise $\{x_1,x_4,x_8\}$ is a good stable set of $G$.

$Q_{28}$ is a $q$-clique, for otherwise $\{x_3,x_5,x_7\}$ is a good stable set of $G$.

$Q_{58}$ is a $q$-clique, for otherwise $\{x_2,x_4\}$ is a good stable set of $G$.

$Q_{45}$ is a $q$-clique, for otherwise $\{x_3,x_7,x_8\}$ is a good stable set of $G$.

\noindent{Now}  we claim that $Q_{47}$ is not a $q$-clique. Suppose not. Then the above properties imply that there is an integer $a$ with $1\le a\le q-1$ such that
$|Q_2|=|Q_5|=|Q_7|=a$ and $|Q_3|=|Q_4|=|Q_8|=q-a$. Since $|Q_{346}|=|Q_6|+2(q-a)\le q$, we have $|Q_6|\le 2a-q$.
Also, since $|Q_{127}|=|Q_1|+2a\le q$, we have $|Q_1|\le q-2a$. However, $2\le |Q_{16}|\le (q-2a)+(2a-q) =0$ which is a contradiction.
So $Q_{47}$ is not a $q$-clique. Then $\{x_2,x_5,x_6\}$ is a good stable set of $G$, and we can conclude the theorem with Theorem~\ref{thm:tools}(\ref{goods}).
\end{proof}

\begin{theorem} \label{thm:54-G7}
Let $G$ be a clique expansion of $G_7$, and assume that every induced subgraph
$G'$ of $G$ satisfies $\chi(G')\le \lceil\frac{5}{4}\omega(G')\rceil$.
Then $\chi(G)\le \lceil\frac{5}{4}\omega(G)\rceil$.
\end{theorem}
\begin{proof}
Since $G$ is a clique expansion of $G_7$ there is a partition of $V(G)$ into
eight non-empty cliques $Q_1, ..., Q_8$, where $Q_i$ corresponds to
the vertex $x_i$ of $G_7$.  We write, e.g., $Q_{123}$ instead of
$Q_1\cup Q_2\cup Q_3$, etc.  Let $q=\omega(G)$.  For each
$i\in\{1,...,8\}$ we call $x_i$ one vertex of $Q_i$.  Moreover if
$|Q_i|\ge 2$ we call $x'_i$ one vertex of $Q_i\setminus\{x_i\}$, and
if $|Q_i|\ge 3$ we call $x''_i$ one vertex of $Q_i\setminus\{x_i,
x'_i\}$.

Suppose that $|Q_7|\le |Q_2|$.  By hypothesis we can color $G\setminus
Q_7$ with $\lceil\frac{5}{4}q\rceil$ colors.  Since $Q_2$ is complete
to $Q_1\cup Q_5$, which is equal to $N(Q_7)$, we can extend this
coloring to $Q_7$, using for $Q_7$ the colors used for $Q_2$.
Therefore let us assume that $|Q_7| > |Q_2|$; and similarly, that
$|Q_8| > |Q_5|$.  It follows that $|Q_{25}|<|Q_{57}|$, so $Q_{25}$ is
not a $q$-clique of $G$.  Likewise $Q_{14}$ is not a $q$-clique of
$G$.  Therefore all $q$-cliques of $G$ are in the set ${\cal
Q}=\{Q_{123}, Q_{456}, Q_{36}, Q_{17}, Q_{57}, Q_{28},  Q_{48}\}$.

If $Q_{123}$ is not a $q$-clique, then $\{x_6, x_7, x_8\}$
is a good stable set of $G$, and we can conclude
using Theorem~\ref{thm:tools}(\ref{goods}).  Therefore we may assume
that $Q_{123}$, and similarly $Q_{456}$, is a $q$-clique of $G$.

If $Q_{36}$ is not a $q$-clique, then $\{x_1, x_5, x_8\}$
is a good  stable set of $G$, and we can conclude
using Theorem~\ref{thm:tools}(\ref{goods}).  Therefore we may assume
that $Q_{36}$ is a $q$-clique of $G$.

If $Q_{17}$ is not a $q$-clique, then $\{x_3, x_5, x_8\}$
is a good stable set of $G$, and we can conclude
using Theorem~\ref{thm:tools}(\ref{goods}).  Therefore we may assume
that $Q_{17}$, and similarly each of $Q_{57}$, $Q_{28}$ and $Q_{48}$, is a
$q$-clique of $G$.

Hence $\cal Q$ is precisely the set of all $q$-cliques of $G$.  It
follows that there are integers $a,b,c$ with $a=|Q_1]$, $b=|Q_2|$,
$c=|Q_3|$, $a+b+c=q$, and then $|Q_7|=q-a$, $|Q_5|=a$, $|Q_8|=q-b$,
$|Q_4|=b$, hence $|Q_6|=c$.  Since $q=|Q_{36}|=2c$, it must be that
$q$ is even and $c=\frac{q}{2}$, so $|Q_3|= |Q_6|=\frac{q}{2}$.

Since each of $Q_1, Q_2, Q_3$ is non-empty we have $q\ge 3$, and since
$q$ is even, $q\ge 4$.  Hence $|Q_3|, |Q_6|\ge 2$ (so the vertices
$x'_3$ and $x'_6$ exist).  Since $Q_2$ and $Q_3$ are non-empty, and
$|Q_3|=\frac{q}{2}$, we have $a< \frac{q}{2}$, so $|Q_7|=q-a >
\frac{q}{2}$, so $|Q_7|\ge 3$ (and so the vertices $x'_7$ and $x''_7$
exist).  Likewise $|Q_8|\ge 3$ (and so the vertices $x'_8$ and $x''_8$
exist).  We observe that the clique $Q_{14}$ satisfies $|Q_{14}|=a+b=
\frac{q}{2}\le q-2$ since $q\ge 4$.  Likewise $|Q_{25}|\le q-2$.

Now consider the five stable sets $\{x_3, x_4, x_7\}$, $\{x_1, x_6,
x_8\}$, $\{x'_3, x_5, x'_8\}$, $\{x'_6,$ $x_2, x'_7\}$ and $\{x''_7,
x''_8\}$.  It is easy to see that their union $U$ meets every
$q$-clique (every member of $\cal Q$) four times, and that it meets
each of $Q_{14}$ and $Q_{25}$ twice.  It follows (since $|Q_{14}|,
|Q_{25}|\le q-2$) that $\omega(G\setminus U)=q-4$, and we can conclude
using Theorem~\ref{thm:tools}(\ref{fives}).
\end{proof}

\begin{theorem}\label{thm:54-G8910}
Let $G$ be a clique expansion of either $G_8$, $G_9$ or $G_{10}$, and assume that every induced subgraph
$G'$ of $G$ satisfies $\chi(G')\le \lceil\frac{5}{4}\omega(G')\rceil$.
Then $\chi(G)\le \lceil\frac{5}{4}\omega(G)\rceil$.
\end{theorem}
\begin{proof}
Let $q=\omega(G)$.

\smallskip

\noindent{(I)} First suppose that $G$ is a clique expansion of $G_8$. So there is a partition of $V(G)$ into
eight non-empty cliques $Q_1, ..., Q_8$, where $Q_i$ corresponds to
the vertex $x_i$ of $G_8$.  We write, e.g., $Q_{123}$ instead of
$Q_1\cup Q_2\cup Q_3$, etc.  For each
$i\in\{1,...,8\}$ we call $x_i$ one vertex of $Q_i$.

Suppose that $|Q_2|\le |Q_7|$.  By hypothesis we can color $G\setminus
Q_2$ with $\lceil\frac{5}{4}q\rceil$ colors.  Since $Q_7$ is complete
to $Q_1\cup Q_3\cup Q_8$, which is equal to $N(Q_2)$, we can extend this
coloring to $Q_2$, using for $Q_2$ the colors used for $Q_7$.
Therefore let us assume that $|Q_2| > |Q_7|$; and similarly, that
$|Q_3| > |Q_8|$.  It follows that $|Q_{28}|>|Q_{78}|$, so $Q_{78}$ is
not a $q$-clique of $G$.  Likewise $|Q_{23}|>|Q_{73}|$, so $Q_{73}$ is
not a $q$-clique of $G$, and similarly $Q_{28}$ is not a $q$-clique.

If $Q_{45}$ is not a $q$-clique, then $\{x_1, x_3, x_8\}$ is a good stable set of $G$, and we can
conclude with Theorem~\ref{thm:tools}(\ref{goods}). Hence we may assume that $Q_{45}$,
and similarly $Q_{16}$, is a $q$-clique. Also $Q_{12}$ is a $q$-clique, for otherwise
$\{x_3, x_5, x_6\}$ is a good stable set, and similarly $Q_{34}$ is a $q$-clique.

Now we claim that $Q_{23}$ is not a $q$-clique of $G$. Suppose not. Then the above properties imply that there is an integer $a$ with $1\le a\le q-1$ such that $|Q_1|=|Q_3|=|Q_5|=a$ and $|Q_2|=|Q_4|=|Q_6|=q-a$. However we have $q\ge |Q_{157}|>2a$
and $q\ge |Q_{468}|>2(q-a)$, hence $2q > 2a + 2(q-a)$, a contradiction. So $Q_{23}$ is not a $q$-clique of $G$. But, then $\{x_1,x_4\}$ is a good stable set of $G$, and we can conclude the theorem with Theorem~\ref{thm:tools}(\ref{goods}).

\medskip

\noindent (II) Now suppose that $G$ is a clique expansion of $G_9$. Let $Q_9$ be the clique that corresponds to $x_9$ in the clique expansion. As in the case of $G_8$ we may assume that $Q_{28}$, $Q_{37}$ and $Q_{78}$ are not $q$-cliques.
Likewise we may assume that $|Q_9|>|Q_5|$ (for otherwise any $\lceil\frac{5}{4}q\rceil$-coloring of
$G\setminus Q_9$ can be extended to $Q_9$), and consequently $Q_{45}$ is not a $q$-clique; and
similarly $Q_{16}$ is not a $q$-clique.

Then $Q_{19}$ is a $q$-clique, for otherwise $\{x_2,x_4,x_7\}$ is a good stable set, and similarly
$Q_{49}$ is a $q$-clique. Also $Q_{12}$ is a $q$-clique, for otherwise $\{x_3,x_5,x_6,x_9\}$ is
a good stable set; and similarly $Q_{34}$ is a $q$-clique. And $Q_{23}$ is a $q$-clique, for
otherwise $\{x_1, x_4\}$ is a $q$-clique.

The  properties given in the preceding paragraph imply that $q$ is even and that
$|Q_1|= |Q_2|=|Q_3|=|Q_4|=|Q_9|=\frac{q}{2}$. We now distinguish two cases.

First suppose that $q=4k$ for some $k\ge 1$. Hence $\lceil\frac{5}{4}q\rceil=5k$.
Let $A,B,C,D,E$ be five disjoint sets of colors, each of size $k$.
We color the vertices in $Q_1$ with the colors from $A\cup B$, the vertices in $Q_2$
with $C\cup D$, the vertices in $Q_3$ with $E\cup A$, the vertices in $Q_4$
with $B\cup C$, and the vertices in $Q_9$ with $D\cup E$. Thus we obtain a $5k$-coloring
of $G[Q_1\cup Q_2\cup Q_3\cup Q_4\cup Q_9]$. We can extend it to the rest of the graph as follows.
Since $Q_{157}$ is a clique, and $|Q_1|=\frac{q}{2}=2k$, we have $|Q_5|+|Q_7|\le 2k$, hence
either $|Q_5|\le k$ or $|Q_7|\le k$. Likewise, we have either $|Q_6|\le k$ or $|Q_8|\le k$.
This yields (up to symmetry) three possibilities: \\
(i) $|Q_5|\le k$ and $|Q_6|\le k$. Then we can color $Q_5$ with colors form $E$, $Q_6$ with colors
from $D$, $Q_7$ with colors from $C\cup D$, and $Q_8$ with colors from $A\cup E$.
\\
(ii) $|Q_5|\le k$ and $|Q_8|\le k$. Then we can color $Q_5$ with colors form $E$, $Q_6$ with colors
from $D\cup E$, $Q_7$ with colors from $C\cup D$, and $Q_8$ with colors from $A$.
(The case where $|Q_6|\le k$ and $|Q_7|\le k$ is symmetric.)
\\
(iii) $|Q_7|\le k$ and $|Q_8|\le k$. Then we can color $Q_5$ and $Q_6$ with colors from $D\cup E$,
$Q_7$ with colors from $C$, and $Q_8$ with colors from $A$.

Now suppose that $q=4k+2$ for some $k\ge 1$. Hence $\lceil\frac{5}{4}q\rceil=5k+3$.
Let $A,B,C,D,E$ and $\{z\}$ be six disjoint sets of colors, with $|A|=|B|=|C|=k$ and
$|D|=|E|=k+1$. So these are $5k+3$ colors.
We color the vertices in $Q_1$ with the colors from $C\cup D$, the vertices in $Q_2$
with $A\cup E$, the vertices in $Q_3$ with $B\cup D$, the vertices in $Q_4$
with $C\cup E$, and the vertices in $Q_9$ with $A\cup B\cup\{z\}$. Thus we obtain a $5k+3$-coloring
of $G[Q_1\cup Q_2\cup Q_3\cup Q_4\cup Q_9]$.
We can extend it to the rest of the graph as follows.
Since $Q_{157}$ is a clique, and $|Q_1|=\frac{q}{2}=2k+1$, we have $|Q_5|+|Q_7|\le 2k+1$, hence
either $|Q_5|\le k$ or $|Q_7|\le k$ (and in any case $\max\{|Q_5|,|Q_7|\}\le 2k$).
Likewise, we have either $|Q_6|\le k$ or $|Q_8|\le k$ (and $\max\{|Q_6|,|Q_8|\}\le 2k$).
This yields (up to symmetry) three possibilities: \\
(i) $|Q_5|\le k$ and $|Q_6|\le k$. Then we can color $Q_5$ with colors form $B$, $Q_6$ with colors
from $A$, $Q_7$ with colors from $A\cup E$, and $Q_8$ with colors from $B\cup D$.
\\
(ii) $|Q_5|\le k$ and $|Q_8|\le k$. Then we can color $Q_5$ with colors form $B$, $Q_6$ with colors
from $A\cup B$, $Q_7$ with colors from $A\cup E$, and $Q_8$ with colors from $D$.
(The case where $|Q_6|\le k$ and $|Q_7|\le k$ is symmetric.)
\\
(iii) $|Q_7|\le k$ and $|Q_8|\le k$. Then we can color $Q_5$ and $Q_6$ with colors from $A\cup B$,
$Q_7$ with colors from $E$, and $Q_8$ with colors from $D$.

\medskip

\noindent (III) Finally suppose that $G$ is a clique expansion of $G_{10}$.
We view $G_{10}$ as the graph with nine vertices $u_1, ..., u_9$ and edges $u_iu_{i+1}$
and $u_iu_{i+3}$ for each $i$ modulo~$9$. For each $i$ let $Q_i$ be the clique of $G$ that
corresponds to $u_i$, and let $x_i$ be one vertex of $Q_i$. As usual we write e.g.~$Q_{12}$ instead
of $Q_1\cup Q_2$, etc. We make two observations. \\
\emph{Observation 1}: If for some $i$ the three cliques $Q_{i,i+1}$, $Q_{i+1,i+2}$ and $Q_{i+2,i+3}$
are not $q$-cliques, then $\{x_{i+4}, x_{i+6}, x_{i+8}\}$ is a good stable set of $G$, and we can
conclude using Theorem~\ref{thm:tools}(\ref{goods}). $\diamond$ \\
\emph{Observation 2}: If for some $i$ we have $|Q_{i-1}|\le \frac{q}{3}$ and $|Q_{i+1}|\le \frac{q}{3}$,
then $|Q_i|\ge \frac{2q}{3}$. Indeed suppose (for $i=1$) that
$|Q_9|\le \frac{q}{3}$, $|Q_2|\le \frac{q}{3}$ and $|Q_1|< \frac{2q}{3}$. Then $Q_{91}$ and
$Q_{12}$ are not $q$-cliques, so, by Observation~1, we may assume that $Q_{89}$ and
$Q_{23}$ are $q$-cliques. Hence $|Q_8|\ge \frac{2q}{3}$, and consequently
$|Q_5|\le \frac{q}{3}$ and $|Q_7|\le \frac{q}{3}$; and similarly $|Q_3|\ge \frac{2q}{3}$, and consequently
$|Q_4|\le \frac{q}{3}$ and $|Q_6|\le \frac{q}{3}$. But then $Q_{45}$, $Q_{56}$ and $Q_{67}$ are
not $q$-cliques, so we can conclude as in Observation~1. $\diamond$

Now, since $Q_{147}$ is a clique, we have $|Q_i|\le \frac{q}{3}$ for some $i\in\{1,4,7\}$;
and similarly $|Q_j|\le \frac{q}{3}$ for some $j\in\{2,5,8\}$, and $|Q_k|\le \frac{q}{3}$ for
some $k\in\{3,6,9\}$. Up to symmetry this implies one the following three cases: \\
(a) $|Q_1|, |Q_2|, |Q_3|\le \frac{q}{3}$. Then we can conclude using Observation~2. \\
(b) $|Q_1|, |Q_2|, |Q_6|\le \frac{q}{3}$. Then $Q_{12}$ is not a $q$-clique, so, by Observation~1,
we may assume that one of $Q_{91}$ and $Q_{23}$, say $Q_{91}$ is a $q$-clique. Hence
$|Q_9|\ge \frac{2q}{3}$, and consequently $|Q_3|\le \frac{q}{3}$. But then we are in case~(a) again. \\
(c) $|Q_1|, |Q_3|, |Q_5|\le \frac{q}{3}$. By Observation~2 we have $|Q_2|\ge \frac{2q}{3}$ and
$|Q_4|\ge \frac{2q}{3}$, and consequently $|Q_8|\le \frac{q}{3}$ and $|Q_7|\le \frac{q}{3}$.
Then $Q_7$, $Q_8$ and $Q_3$ are like in case~(b).

This completes the proof of the theorem.
\end{proof}

\subsection{Coloring the graph class $\cal{H^*}$}

Recall that   $\cal{H}^*$ is the class of graphs $G \in \cal{H}$ such that, with the notation as in Section~1, the five sets $A_1,A_2,\ldots, A_5$, and the vertex-set of each component of $G[A_7]$ are cliques.

\begin{theorem}\label{thm:54-H}
Let $G\in \cal{H^*}$  and assume that every induced subgraph
$G'$ of $G$ satisfies $\chi(G')\le \lceil\frac{5}{4}\omega(G')\rceil$.
Then $\chi(G)\le \lceil\frac{5}{4}\omega(G)\rceil$.
\end{theorem}
\begin{proof} Let $q=\omega(G)$. Let $T_1,T_2,\ldots,T_k$ be the components of $G[A_7]$.
For each $i\in \{1,\ldots,5\}$ and  for each $j\in \{1,\ldots,k\}$: let $x_i$ be one vertex of $A_i$,   and let $t_j$ be one vertex of $V(T_j)$.  Moreover  if $|A_i|\geq 2$ we call $x_i'$ one vertex of $A_i\setminus \{x_i\}$, if $|V(T_i)|\ge 2$ we call $t_i^1$ one vertex of $V(T_i)\setminus\{t_i\}$,
if $|V(T_i)|\ge 3$ we call $t_i^2$ one vertex of $V(T_i)\setminus\{t_i, t_i^1\}$, and if $|V(T_i)|\ge 4$ we call $t_i^3$ one vertex of $V(T_i)\setminus\{t_i, t_i^1,t_i^2\}$.

Suppose that $|A_2| \leq \omega(G[A_6])$. Then by hypothesis, $G\sm A_2$ can be colored with  $\lceil\frac{5}{4}q\rceil$ colors, and since $A_6$ is complete to $A_1\cup A_3$ which is equal to $N(A_2)$,  we  can extend this coloring to $A_2$ by using the colors of $A_6$ on $A_2$.
So we may assume that $|A_2| > \omega(G[A_6])$. Likewise, $|A_5|> \omega(G[A_6])$. So it follows that no clique of $A_1\cup A_6$ is a $q$-clique of $G$.

Now consider the stable set $S:= \{x_2,x_5,t_1,\ldots,t_k\}$. We may assume that $S$ is not a good stable set of $G$ (otherwise, we can conclude with Theorem~\ref{thm:tools}(\ref{goods})). So there is a maximum clique $Q$ of $G$ contained in $A_3\cup A_4\cup A_6$.  Further, it follows that
for every maximum clique $Q$ of $G$ with $Q\cap S =\emptyset$, we have $A_3\cup A_4\subset Q$.

If $A_1\cup A_2$ is not a $q$-clique, then $\{x_3,x_5, t_1,\ldots,t_k\}$ is a good stable set of $G$, and we can conclude using Theorem~\ref{thm:tools}(\ref{goods}).  So we may assume that $A_1\cup A_2$ is a $q$-clique of $G$. Likewise, $A_1\cup A_5$ is a $q$-clique of $G$.

If $A_2\cup A_3$ is not a $q$-clique, then $\{x_1,x_4, t_1,\ldots,t_k\}$ is a good stable set of $G$, and we can conclude using Theorem~\ref{thm:tools}(\ref{goods}).  So we may assume that $A_2\cup A_3$ is a $q$-clique of $G$. Likewise, $A_4\cup A_5$ is a $q$-clique of $G$.

The above properties imply that there is an integer $a$ with $1\le a\le q-1$ such that
$|A_1|=|A_3|=|A_4|=a$ and $|A_2|=|A_5|=q-a$.  Moreover, every $q$-clique of $G$ either contains $A_i\cup A_{i+1}$, for some $i\in \{1,\ldots,5\}$, $i$ modulo $5$, or contains $T_j$, for some $j\in \{1,\ldots,k\}$.

Now if $|V(T_j)|\leq 2a$, for some $j$, then  by hypothesis, $G\sm V(T_j)$ can be colored with  $\lceil\frac{5}{4}q\rceil$ colors. Since $|A_3\cup A_4|=2a$, $V(T_j)$ is anticomplete to $A_3\cup A_4$, $N(V(T_j))\subseteq A_6$, and since $A_6$ is complete to $A_3\cup A_4$,  we  can extend this coloring to $V(T_j)$ by using the colors of $A_3\cup A_4$ on $V(T_j)$. So, we may assume that, for each $j\in\{1,\ldots,k\}$, $|V(T_j)| > 2a$.

If $a=1$, then $\mbox{deg}_G(x_2) =2 \leq \lceil\frac{5}{4}q\rceil-1$, and we can conclude with Theorem~\ref{thm:tools}(\ref{degq}). So we may assume that $a\geq 2$.

Thus for each $j\in\{1,\ldots,k\}$, we have $|V(T_j)| \geq 4$. Also, since $q-a>\omega(G[A_6])$, we have $q-a\geq 2$.

Now consider the five stable sets $\{x_1,x_3,t_1,t_2\ldots,t_k\}$, $\{x_3',x_5,t_1^1,$ $t_2^1\ldots,t_k^1\}$, $\{x_2,x_5',t_1^2,t_2^2,\ldots,t_k^2\}$, $\{x_2',x_4,t_1^3,t_2^3\ldots,t_k^3\}$, and $\{x_1',x_4'\}$. It is easy to see that their union $U$ meets every $q$-clique of $G$  four times. It follows that $\omega(G\setminus U) = q-4$, and we can conclude using Theorem~\ref{thm:tools}(\ref{fives}). \end{proof}

\noindent{\bf Proof of Theorem~\ref{thm:54bound}}. If $G$ is perfect, then $\chi(G) = \omega(G)$ and the theorem holds. So we may assume that $G$ is not perfect, and that $G$ is connected. Since a $P_5$-free graph contains no hole of length at least $7$, and a gem-free graph contains no antihole of length at least $7$, it follows from the Strong Perfect Graph Theorem \cite{SPGT} that $G$ contains a hole of length $5$. That is, $G$ contains a $C_5$ as an induced subgraph.
Now, the result follows directly from Theorem~\ref{thm:struc},  Lemma~\ref{lem:red2}, and Theorems~\ref{thm:54-G1-G6}, \ref{thm:54-G7}, \ref{thm:54-G8910} and \ref{thm:54-H}. \hfill{$\Box$}

\medskip
 \noindent{\bf Acknowledgements}. The second author would like to  thank Mathew~C.~Francis for helpful discussions.


\small
\begin{thebibliography}{99}

\bibitem{BK2005}
Brandst\"adt, A., Kratsch, D.: On the structure of ($P_5$,\,gem)-free graphs. Discrete Applied Mathematics 145 (2005) 155--166.



\bibitem{CHM} Cameron, K., Huang, S., Merkel, O.:
An improved $\chi$-bound for $(P_5,gem)$-free graphs. Private
communication.

\bibitem{CKS}
Choudum, S. A., Karthick, T., Shalu, M. A.: Perfect coloring and
linearly $\chi$-bounded $P_6$-free graphs.  Journal of Graph Theory
{54} (2007) 293--306.

\bibitem{SPGT}
Chudnovsky, M., Seymour, P., Robertson, N., Thomas, R.: The strong
perfect graph theorem.  Annals of Mathematics {164} (2006) 51--229.

\bibitem{Chud-Siva} Chudnovsky, M.,  Sivaraman, V.: Perfect divisibility and 2-divisibility. Journal of Graph Theory, available online.

\bibitem{Fouquet}
Fouquet, J. L., Giakoumakis, V.,  Maire, F., Thuillier, H.:  On graphs
without $P_5$ and $\overline{P_5}$.  Discrete Mathematics
146 (1995) 33--44.

\bibitem{GHM}
Gravier, S., Ho\`ang, C.T., Maffray, F.: Coloring the hypergraph of
maximal cliques of a graph with no long path.  Discrete Mathematics
272 (2003) 285--290.

\bibitem{Gyarfas}
Gy\'{a}rf\'{a}s, A. Problems from the world surrounding perfect
graphs.  Zastosowania Matematyki Applicationes Mathematicae {19}
(1987) 413--441.


\bibitem{KM2018}
Karthick, T., Maffray, F.:  Coloring (gem,\,co-gem)-free graphs.  Journal of Graph Theory  89 (2018) 288--303.

\bibitem{KM-arxiv}
Karthick, T., Maffray, F.:   Square-free graphs with no six-vertex induced path. Extended abstract in: Proceedings of 10th International Colloquium on Graph Theory and Combinatorics (ICGT-2018), Lyon, France (2018). Available on arXiv:1805.05007 [cs.DM].

\bibitem{KPT-P5}
 Kierstead, H. A., Penrice, S.G., Trotter, W.T.: On-line and first-fit
coloring of graphs that do not induce $P_5$.   SIAM Journal of
Discrete Mathematics 8 (1995) 485--498.


\bibitem{Reed}
Reed. B.:  $\omega, \Delta$ and $\chi$.   Journal of Graph Theory
{27} (1998) 177--212.



\end{thebibliography}
\end{document}